%% file: infinitecobordism.tex
\documentclass[pdflatex, a4paper]{amsart}

\usepackage[ngerman, english]{babel}
\usepackage[macce]{inputenc}
\usepackage{lmodern}
\usepackage{amsmath}
\usepackage[usenames,dvipsnames]{color}
\usepackage{amssymb}
\usepackage{mathrsfs}
\usepackage{amsthm}
\usepackage{tikz}
\usepackage{tikz-cd}
\usepackage{braket}
\usepackage{mathtools}
\usepackage{graphicx}

\definecolor{darkblue}{rgb}{0,0,.5}
\definecolor{darkgreen}{rgb}{0,0.5,0}
\usepackage[colorlinks=true, linkcolor=darkblue, citecolor = darkgreen]{hyperref}
\usepackage{microtype}

\newtheorem*{thm}{Main Theorem}
\newtheorem{Theorem}{Theorem}
\newtheorem{Lemma}{Lemma}
\newtheorem{Corollary}{Corollary}
\newtheorem{Proposition}{Proposition}

\theoremstyle{definition}
\newtheorem{Definition}{Definition}

\newcommand{\N}{\mathbb N}

\newcommand{\R}{\mathbb R}

\DeclareMathOperator{\colim}{colim}

\newcommand{\catname}[1]{\mathbf{#1}}

\title{On the infinite loop space structure of the cobordism category}
\author{Hoang Kim Nguyen}

\begin{document}
	\begin{abstract}
		We show an equivalence of infinite loop spaces between the classifying space of the cobordism category, with infinite loop space structure induced by taking disjoint union of manifolds, and the infinite loop space associated to the Madsen-Tillmann spectrum.
	\end{abstract}
\maketitle
\tableofcontents

\section{Introduction} 
\label{sec:introduction}

In this article we show that there is an equivalence of infinite loop spaces between the classifying space of the $d$-dimensional cobordism category $B \mathrm{Cob}_\theta(d)$ and the 0-th space of the shifted Madsen-Tillmann spectrum $MT \theta(d)[1]$. This extends a result by Galatius, Madsen, Tillmann and Weiss \cite{gmtw}, who showed an equivalence of topological spaces
\[
   B\mathrm{Cob}_\theta(d) \simeq MT \theta(d)[1]_0.
\]
Note that both spaces in the equivalence above admit an infinite loop space structure. The symmetric monoidal structure on the cobordism category, given by disjoint union of manifolds, induces an infinite loop space structure on $B\mathrm{Cob}_\theta(d)$, while the infinite loop space structure on $MT\theta(d)[1]_0$ comes from it being the 0-th space of an $\Omega$-spectrum. We will show that the equivalence of \cite{gmtw} actually extends to an equivalence of infinite loop spaces with the above mentioned infinite loop space structures. 

In more detail, our proof will rely on certain spaces of manifolds introduced by Galatius and Randal-Williams \cite{galatiusrandal}, which form an $\Omega$-spectrum denoted here by $\psi_\theta$. Using these spaces, they obtain a new proof of the result of \cite{gmtw}, which we record as the following theorem.

\begin{Theorem}\label{thm1}
	There are weak homotopy equivalences of spaces
	\[
	 B\mathrm{Cob}_\theta(d) \simeq \psi_{\theta,0} \simeq MT\theta(d)[1]_0.
\]
\end{Theorem}
In this article, we will show that the equivalences above come from equivalences of spectra.

Instead of directly constructing an equivalence of spectra, our strategy will be to construct $\Gamma$-spaces $\Gamma \mathrm{Cob}_\theta(d)$ and $\Gamma \psi_\theta$ with underlying spaces $B\mathrm{Cob}_\theta(d)$ and $\psi_{\theta,0}$ respectively and we show that $\Gamma \psi_\theta$  is a model for the connective cover of the spectrum $\psi_\theta$, denoted by $\psi_{\theta, \geq 0}$. This $\Gamma$-structure will be induced by taking disjoint union of manifolds. We then show that their associated spectra have the stable homotopy type of the connective cover of the shifted Madsen-Tillmann spectrum denoted by $MT\theta(d)[1]_{\geq 0}$, by constructing a $\Gamma$-space model for $MT\theta(d)[1]_{\geq 0}$ and exhibiting an equivalence of $\Gamma$-spaces. But more is true, we will see that the equivalences of theorem \ref{thm1} are the components of this equivalence of $\Gamma$-spaces and hence the main result of this article will be the following.

\begin{thm}
	There are stable equivalences of spectra
	\[
	\mathbf B \Gamma \mathrm{Cob}_\theta(d) \simeq \psi_{\theta,\geq 0} \simeq MT\theta(d)[1]_{\geq 0}
\]
such that the induced weak equivalences of spaces
\[
\Omega^\infty \mathbf B \Gamma \mathrm{Cob}_\theta(d) \simeq \Omega^\infty \psi_\theta \simeq \Omega^\infty MT \theta(d)[1]_{\geq 0}
\]
are equivalent to the weak equivalences of theorem \ref{thm1}.
\end{thm}
Here, $\mathbf B\Gamma \mathrm{Cob}_\theta(d)$ is the spectrum associated to the symmetric monoidal category $\mathrm{Cob}_\theta(d)$.
We would like to mention that a similar argument has been given by Madsen and Tillmann in \cite{madsentillmann} for the case $d=1$.

This article is organized as follows. In the next section we recall some basic notions on spectra and $\Gamma$-spaces. This will also serve to fix notation and language. In section \ref{sec:spacesofmanifolds} and \ref{sec:the_weak_homotopy_type_of_psi__theta_infty_1} we review the proof of theorem \ref{thm1} of \cite{galatiusrandal}. In section \ref{sec:gamma_space_models} we will construct $\Gamma$-space models for the spectra $\psi_\theta$ and $MT\theta(d)$ and in section \ref{sec:equivalence_of_gamma_space_models} we will show that these $\Gamma$-spaces are equivalent. Finally in section \ref{sec:the_cobordism_category}, we will relate these $\Gamma$-spaces to the cobordism category with its infinite loop space structure induced by taking disjoint union of manifolds.\\
\\
\textsc{Acknowledgements} The author would like to thank Ulrich Bunke, John Lind and George Raptis for may helpful conversations and especially Ulrich Bunke for his encouragement to write this article. Furthermore, the author gratefully acknowledges support from the Deutsche Forschungsgesellschaft through the SFB 1085 \emph{Higher Invariants}.

\section{Conventions on spectra and $\Gamma$-spaces} 
\label{sec:spectra}

By a space we mean a compactly generated weak Hausdorff space. We denote by $\catname{S}$ the category of spaces and by $\catname S_\ast$ the category of based spaces. We fix a model for the circle by setting $S^1:= \R \cup \{\infty\}$.

We will work with the Bousfield-Friedlander model of sequential spectra, see \cite{bousfield} or \cite{mmss}. Recall that a spectrum $E$ is a sequence of based spaces $E_n \in \catname S_\ast$, $n\in \N$ together with structure maps
\[
s_n: S^1 \wedge E_n \to E_{n+1}.
\]
A map of spectra $f: E \to F$ is a sequence of maps $f_n : E_n \to F_n$ commuting with the structure maps. We denote by $\catname{Spt}$ the category of spectra. A stable equivalence is a map of spectra inducing isomorphisms on stable homotopy groups. An $\Omega$-spectrum is a spectrum $E$, where the adjoints of the structure maps $\Sigma E_n \to E_{n+1}$ are weak homotopy equivalences. There is a model structure on $\catname{Spt}$ with weak equivalences the stable equivalences and fibrant objects the $\Omega$-spectra. Moreover, a stable equivalence between $\Omega$-spectra is a levelwise weak homotopy equivalence. We obtain a Quillen adjunction
\[	
\Sigma^\infty : \catname S_\ast \leftrightarrow \catname{Spt} : \Omega^\infty
\]
where $\Sigma^\infty$ takes a based space to its suspension spectrum and $\Omega^\infty$ assigns to a spectrum its 0-th space.

A spectrum $E$ is called connective, if its negative homotopy groups vanish. In case $E$ is an $\Omega$-spectrum this is equivalnet to $E_n$ being $(n-1)$-connected for all $n \in \N$. Note that a map $f: E \to F$ between connective $\Omega$-spectra is a stable equivalence if and only if $f_0:E_0 \to F_0$ is a weak homotopy equivalence. We denote by $\catname{Spt}_{\geq 0}$ the full subcategory of connective spectra, it is a reflective subcategory and we denote the left adjoint of the inclusion by
\[
(-)_{\geq 0}: \catname{Spt} \to \catname{Spt}_{\geq 0}.
\]

We will need two operations on spectra. The first one is the shift functor
\[
(-)[1]: \catname{Spt} \to \catname{Spt}
\]
defined on a spectrum $E$ by setting $E[1]_n = E_{n+1}$ and obvious structure maps. The second operation is the loop functor
\[
\Omega: \catname{Spt} \to \catname{Spt}
\]
defined by $(\Omega E)_n = \Omega (E_n)$ and looping the structure maps.

We recall Segal's infinite loop space machine \cite{segal}, which provides many examples of connective spectra. We denote by $\Gamma^{op}$ the skeleton of the category of finite pointed sets and pointed maps, i.e. its objects are the sets $n_+:= \{\ast,1,\ldots,n\}$. A $\Gamma$-space is a functor
\[
\Gamma^{op} \to \catname S_\ast
\]
and we denote by $\Gamma \catname S_\ast$ the category of $\Gamma$-spaces and natural transformations.

There are distinguished maps $\rho_i: m_+ \to 1_+$ defined by $\rho_i(k)=\ast$ if $k \neq i$ and $\rho_i (i)=1$. Let $A \in \Gamma \catname S_\ast$, the Segal map is the map
\[
A(m_+)\xrightarrow{\prod_{i=1}^m \rho_i} \prod_m A(1_+).
\]
A $\Gamma$-space is called special if the Segal map is a weak homotopy equivalence. If $A \in \Gamma \catname S_\ast$ is special, the set $\pi_0(A(1_+))$ is a monoid with multiplication induced by the span
\[
A(1_+) \leftarrow A(2_+) \xrightarrow{\simeq} A(1_+)\times A(1_+)
\]
where the left map is the map  sending $i \mapsto 1$ for $i=1,2$ and the right map is the Segal map. A special $\Gamma$-space is called very special if this monoid is actually a group.

In \cite{bousfield} Bousfield and Friedlander construct a model structure on $\Gamma \catname S_\ast$ with fibrant objects the very special $\Gamma$-spaces and weak equivalences between fibrant objects levelwise weak equivalences.

There is a functor $\mathbf B: \Gamma \catname S_\ast \to \catname{Spt}$ defined as follows. Denote by $\mathbb S: \Gamma^{op} \to \catname S_\ast$ the inclusion of finite pointed sets into pointed spaces. Given $A\in \Gamma \catname S_\ast$ we have a (enriched) left Kan extension along $\mathbb S$
\[
\begin{tikzcd}
	\Gamma^{op} \rar{A} \dar[swap]{\mathbb S} & \catname S_\ast \\
	\catname S_\ast \urar[dashed] &
\end{tikzcd}
\]
and we denote this left Kan extension by $L_\mathbb S A.$ Now define $\mathbf BA_n:= L_\mathbb S A(S^n)$. The structure maps are given by the image of the identity morphism $S^1 \wedge S^n \to S^1 \wedge S^n$ under the composite map
\begin{align*}
	\catname S_\ast(S^1 \wedge S^n,S^1 \wedge S^n) & \cong S_\ast(S^1, \catname S_\ast (S^n,S^{n+1}))\\
	&\to S_\ast (S^1, \catname S_\ast(L_\mathbb SA(S^n), L_\mathbb S(A(S^{n+1}))))\\
	&\cong \catname S_\ast(S^1 \wedge L_\mathbb SA(S^n),L_\mathbb SA(S^{n+1})).
\end{align*}
By the Barratt-Priddy-Quillen theorem $L_\mathbb S \mathbb S$ is the sphere spectrum, hence the notation.

The functor $\mathbf B$ has a right adjoint $\mathbf A: \catname{Spt}\to \Gamma \catname S_\ast$ given by sending a spectrum $E \in \catname{Spt}$ to the $\Gamma$-space
\[
n_+ \mapsto \catname{Spt}(\mathbb S^{\times n},E)
\]
using the topological enrichment of spectra. Moreover, the adjoint pair $\mathbf B \dashv \mathbf A$ is a Quillen pair which induces an equivalence of categories
\[
\mathrm{Ho}(\Gamma \catname S_\ast) \simeq \mathrm{Ho}(\catname{Spt}_{\geq 0}).
\]
In view of this equivalence we will say that a $\Gamma$-space $A$ is a model for a connective spectrum $E$, if there is a stable equivalence $\mathbb L \mathbf B A \simeq E$, where $\mathbb L\mathbf B$ is the left derived functor.

Finally, the main theorem of Segal \cite{segal} states that $\mathbf B$ sends cofibrant-fibrant $\Gamma$-spaces to connective $\Omega$-spectra.

\section{Recollection on spaces of manifolds} 
\label{sec:spacesofmanifolds}
We recall the spaces of embedded manifolds with tangential structure from \cite{galatiusrandal}. Denote by $Gr_d(\R^n)$ the Grassmannian manifold of $d$-dimensional planes in $\R^n$ and denote $BO(d):= \colim_{n\in \N} Gr_d(\R^n)$ induced by the standard inclusion $\R^n \to \R^{n+1}$. Let $\theta: X \to BO(d)$ be a Serre fibration and let $M \subset \R^n$ be a $d$-dimensional embedded smooth manifold. Then a tangential $\theta$-structure on $M$ is a lift
\[
\begin{tikzcd}
{} & X \dar{\theta}\\
M \urar[dashed] \rar[swap]{\tau_M} & BO(d),
\end{tikzcd}
\]
where $\tau_M$ is the classifying map of the tangent bundle (determined by the embedding). The topological space $\Psi_\theta (\R^n)$ has as underlying set pairs $(M,l)$, where $M$ is a $d$-dimensional smooth manifold without boundary which is closed as a subset of $\R^n$ and $l:M \to X$ is a $\theta$-structure. We refer to \cite{galatiusrandal} and \cite{galatius} for a description of the topology. We will also in general suppress the tangential structure from the notation.

For $0 \leq k \leq n$, we have the subspaces $\psi_\theta (n,k) \subset \Psi_\theta(\R^n)$ of those manifolds $M \subset \R^n$, satisfying
\[
M \subset \R^k \times (-1,1)^{n-k}.
\]
That is, manifolds with $k$ possibly non-compact and $(n-k)$ compact directions. We denote by
\begin{align*}
\Psi_\theta (\R^\infty)&:= \colim_{n\in \N} \Psi_\theta (\R^n)\\
\psi_\theta(\infty,k)&:= \colim_{n\in \N}\psi_\theta(n,k)
\end{align*}
again induced by the standard inclusions. In \cite{bokstedt} it is shown that the spaces $\Psi_\theta(\R^n)$ are metrizable and hence in particular compactly generated weak Hausdorff spaces.

For all $n \in \N$ and $1 \leq k \leq n-1$ we have a map 
\begin{align*}
\R\times \psi_\theta(n,k) &\to \psi_\theta(n,k+1)\\
(t,M) & \mapsto M-t \cdot e_{k+1}.
\end{align*}
This descends to a map $S^1 \wedge \psi_\theta(n,k) \to \psi_\theta(n,k+1)$ when taking as basepoint the empty manifold.

\begin{Theorem}\label{adj}
The adjoint map
\[
\psi_\theta(n,k) \to \Omega \psi_\theta(n,k+1)
\]
is a weak homotopy equivalence.
\end{Theorem}

\begin{proof}
\cite[Theorem 3.20]{galatiusrandal}.
\end{proof}

\begin{Definition}
Let $\psi_\theta$ be the spectrum with $n$-th space given by
\[
(\psi_\theta)_n := \psi_\theta (\infty,n+1)
\]
and structure maps given by the adjoint of the translation.
\end{Definition}

By the above theorem, the spectrum $\psi_\theta$ is an $\Omega$-spectrum.


\section{The weak homotopy type of $\psi_\theta(\infty,1)$} 
\label{sec:the_weak_homotopy_type_of_psi__theta_infty_1}
This section contains a brief review of the main theorem of \cite{gmtw} as proven in \cite{galatiusrandal}. Recall first the construction of the Madsen-Tillmann spectrum $MT \theta(d)$ associated to a Serre fibration $\theta: X \to BO(d)$. Denote by $X(\R^n)$ the pullback
\[
\begin{tikzcd}
X(\R^n)\rar \dar[swap]{\theta_n} & X \dar{\theta}\\
Gr_d(\R^n) \rar & BO(d)
\end{tikzcd}
\]
and by $\gamma_{d,n}^\perp$ the orthogonal complement of the tautological bundle over $Gr_d(\R^n)$. Then define the spectrum $T\theta(d)$ to have as $n$-th space the Thom space of the pullback bundle $T \theta(d)_n:= Th(\theta^\ast_n\gamma_{d,n}^\perp)$. The structure maps are given by 
\[
S^1 \wedge Th(\theta^\ast_n \gamma_{d,n}^\perp)\cong Th(\theta^\ast_n \gamma_{d,n}^\perp \oplus \varepsilon) \to Th (\theta^\ast_{n+1}\gamma_{d,n+1}^\perp)
\]
where $\varepsilon$ denotes the trivial bundle. Then define the Madsen-Tillmann spectrum $MT\theta(d)$ to be a fibrant replacement of the spectrum $T\theta(d)$. Since the adjoints of the structure maps of $T\theta(d)$ are inclusions, we can give an explicit construction of $MT \theta(d)$ as
\[
MT \theta(d)_n := \colim_k \Omega^k T\theta(d)_{n+k}.
\]
Hence we have $\Omega^\infty MT \theta (d) = \colim_k \Omega^k T \theta(d)_k$.

The passage from $MT \theta(d)$ to our spaces of manifolds is as follows. We have a map
\[
Th(\theta^\ast_n \gamma_{d,n}^\perp)\to \Psi_\theta(\R^n)
\]
given by sending an element $(V,u,x)$, where $V \in Gr_d(\R^n)$, $u \in V^\perp$ and $x \in X$, to the translated plane $V-u \in \Psi_\theta (\R^n)$ with constant $\theta$-structure at $x$ and sending the basepoint to the empty manifold.

\begin{Theorem}\label{maingalran2}
The map $Th(\theta^\ast_n \gamma_{d,n}^\perp)\to \Psi_\theta(\R^n)$ is a weak homotopy equivalence.
\end{Theorem}

\begin{proof}
\cite[Theorem 3.22]{galatiusrandal}.
\end{proof}

On the other hand, by theorem \ref{adj} we also have a weak homotopy equivalence
\[
\psi_\theta(n,1) \to \Omega^{n-1} \Psi_\theta(\R^n).
\]
Combining the two equivalences, we obtain
\[
\Omega^{n-1} Th (\theta^\ast_n \gamma_{d,n}^\perp) \xrightarrow{\simeq} \Omega^{n-1} \Psi_\theta(\R^n) \xleftarrow{\simeq} \psi_\theta(n,1).
\]
Now we have a map
\begin{align*}
S^1 \wedge \Psi_\theta(\R^n) &\to \Psi_\theta(\R^{n+1})\\
(t,M)&\mapsto M \times \{t\},
\end{align*}
and we obtain the commutative diagram
\[
\begin{tikzcd}
\Omega^{n-1} Th (\theta^\ast_n \gamma_{d,n}^\perp)\rar{\simeq}\dar \rar{\simeq}& \Omega^{n-1} \Psi_\theta (\R^n) \dar &  \psi_\theta(n,1)\dar \lar[swap]{\simeq}\\ 
\Omega^{n} Th (\theta^\ast_{n+1} \gamma_{d,n+1}^\perp)\rar{\simeq} & \Omega^n \Psi_\theta(\R^{n+1}) & \psi_\theta(n+1,1) \lar[swap]{\simeq}. 
\end{tikzcd}
\]
Finally, letting $n\to \infty$ we can determine the weak homotopy type of $\psi_\theta(\infty,1)$.

\begin{Theorem}\label{maingalran}
There is a weak homotopy equivalence of spaces
\[
\Omega^\infty MT\theta(d)[1] \xrightarrow{\simeq} \colim_{n\in\N}\Omega^{n-1} \Psi_\theta (\R^n) \xleftarrow{\simeq} \psi_\theta(\infty,1).
\]
\end{Theorem}

\section{$\Gamma$-space models for $MT\theta(d)$ and $\psi_\theta$} 
\label{sec:gamma_space_models}

In this section we construct $\Gamma$-space models for the spectra $MT\theta(d)$ and $\psi_\theta$. The comparison of these $\Gamma$-spaces to the respective spectra relies heavily on results of May and Thomason \cite{maythomason}. 

\begin{Lemma}\label{switch}
	Let $A\in \Gamma \catname S_\ast$ then there is a natural map of spectra
	\[
	\mathbf B\Omega A \to \Omega \mathbf BA
\]
which is the identity on 0-th spaces.
\end{Lemma}

\begin{proof}
Since $\mathbb S : \Gamma^{op} \to \catname S_\ast$ is fully faithful, we have a strictly commutative diagram of functors
\[
\begin{tikzcd}
	\Gamma^{op}\rar{\Omega A}\dar[swap]{\mathbb S} & \catname S_\ast\\
	\catname S_\ast \urar[swap]{L_\mathbb S \Omega A}. & 
\end{tikzcd}
\]
Now the composition of the loop functor with the left Kan extension $\Omega L_\mathbb S A$ also gives a strictly commutative diagram
\[
\begin{tikzcd}
	\Gamma^{op}\rar{\Omega A}\dar[swap]{\mathbb S} & \catname S_\ast\\
	\catname S_\ast \urar[swap]{\Omega L_\mathbb S A}. & 
\end{tikzcd}
\]
Hence by the universal property of the left Kan extension we get a natural transformation $\gamma: L_\mathbb S \Omega A\Rightarrow \Omega L_\mathbb S A $. Now the components at the spheres assemble into a map of spectra $\mathbf B\Omega A \to \Omega \mathbf BA$, since by naturality we have a commutative diagram
\[
\begin{tikzcd}
	S^1\wedge L_\mathbb S\Omega A(S^n) \rar \dar[swap]{id \wedge \gamma} & L_\mathbb S\Omega A(S^{n+1})\dar{\gamma}\\
	S^1\wedge \Omega L_\mathbb SA(S^n) \rar & \Omega L_\mathbb SA(S^{n+1}).
\end{tikzcd}
\]
Finally, since $S^0 = 1_+\in \Gamma^{op}$ the map of spectra is the identity on 0-th spaces.
\end{proof}

In general for any $A\in \Gamma \catname S_\ast$ the spectrum $\mathbf B A$ might not have the right stable homotopy type as the functor $\mathbf B$ only preserves weak equivalences between cofibrant objects. However for very special $\Gamma$-spaces, there is a more convenient replacement, which gives the right homotopy type. We record the following fact from May-Thomason \cite{maythomason}, which generalizes a construction of Segal in \cite{segal}.

\begin{Lemma}\label{cofibrant}
There is a functor $W: \Gamma \catname S_\ast \to \Gamma \catname S_\ast$ such that the following holds for all very special $X \in \Gamma \catname S_\ast$.
\begin{itemize}
	\item The spectrum $\mathbf BWX$ is a connective $\Omega$-spectrum.
	\item The $\Gamma$-space $WX$ is very special and there is a weak equivalence $WX \to X$.
	\item If $X,Y$ are very special and there is a weak equivalence $X \simeq Y$ then $\mathbf BWX \simeq \mathbf BWY$.
	\item There is a weak equivalence $W\Omega X \to \Omega W X$.
\end{itemize}
\end{Lemma}

\begin{proof}
	See \cite[Appendix B]{maythomason}.
\end{proof}

The important thing for us will be that if $X\in \Gamma \catname S_\ast$ is very special, then $\mathbf BWX$ has the right stable homotopy type. 

\begin{Lemma}\label{upandacross}
	Let $E^i$, $i \in \N$ be a sequence of connective $\Omega$-spectra together with stable equivalences $f^i: E^i \to \Omega E^{i+1}$. Let $E_0$ be the spectrum with $\left(E_0\right)_n := E^n_0$ and structure maps given by $f^n_0: E^n_0 \to \Omega E^{n+1}_0$. Then there is a natural stable equivalence $E^0 \simeq E_0$.
\end{Lemma}

\begin{proof}
	This is the 'up-and-across lemma` of \cite{maythomason} and \cite{fiedorowicz}.
\end{proof}

Note that in particular $E_0$ is connective. We will encounter the following situation.

\begin{Definition}
A functor $E: \Gamma^{op} \to \catname{Spt}$ is called a $\Gamma$-spectrum. It is called a special $\Gamma$-spectrum if the Segal map
\[
E(m_+) \to \prod_m E(1_+)
\]
is a stable equivalence. Furthermore, we denote by $\Gamma^{(k)}E$ the $\Gamma$-space given by evaluating at the $k$-th space, that is
\[
\Gamma^{(k)}E(m_+):= E(m_+)_k.
\] 
\end{Definition}

The key lemma for showing that we have constructed the right $\Gamma$-spaces will be the following.

\begin{Lemma}\label{key}
Let $E: \Gamma^{op}\to \catname{Spt}$ be projectively fibrant and special. Then the $\Gamma$-space $\Gamma^{(0)}E$ is a model for the connective cover of $E(1_+)$.
\end{Lemma}

\begin{proof}
We first show that the $\Gamma$-spaces $\Gamma^{(k)}E$ are very special. In fact we will only need this for $k=0$, the argument for higher $k$ is completely analoguous. First note that the $\Gamma$-spaces are special, since $E$ is projectively fibrant and thus the Segal map is a levelwise equivalence. It remains to show that $\pi_0\left(\Gamma^{(0)}E(1_+)\right)$ is a group. To this end we compose with the functor $\mathbf A: \catname{Spt}\to \Gamma \catname S_\ast$ to obtain a functor
\[
\mathbf AE: \Gamma^{op} \to \Gamma \catname S_\ast
\]
or equivalently a functor
\[
\mathbf AE: \Gamma^{op}\times \Gamma^{op} \to \catname S_\ast
\]
which we will not distinguish notationally. Observe that we have
\begin{align*}
	\mathbf AE(1_+)(-) &= \mathbf AE(1_+) \ \ \mathrm{and}\\
	\mathbf AE(-)(1_+) &= \Gamma^{(0)}E.
\end{align*}
Now we have the following diagram, where the middle square commutes by functoriality
\[
\begin{tikzcd}
	{}& & \mathbf AE(1_+)(1_+) \times \mathbf AE(1_+)(1_+)\\
	& \mathbf AE(2_+)(2_+) \rar \dar & \mathbf AE(2_+)(1_+)\uar[swap]{\simeq}\dar \\
	\mathbf AE(1_+)(1_+)\times \mathbf AE(1_+)(1_+) & \mathbf AE(1_+)(2_+) \lar{\simeq} \rar & \mathbf AE(1_+)(1_+).
\end{tikzcd}
\]
The right vertical span represents the monoid structure of $\Gamma^{(0)}E$ and the lower horizontal span represents the monoid structure of $\mathbf AE(1_+)$. The commutativity of the middle square is now precisely the statement that one is a homomorphism for the other, thus by the Eckmann-Hilton argument they agree. We now observe that the monoid $\mathbf AE(1_+)$ is actually a group, since $\pi_0(\mathbf AE(1_+)(1_+))$ is the 0-th stable homotopy group of $E(1_+)$. It follows that $\Gamma^{(0)}E$ is very special.

As a next step, we compose with taking connective covers to obtain a special $\Gamma$-spectrum in connective $\Omega$-spectra
\[
E_{\geq 0}: \Gamma^{op}\to \catname{Spt}_{\geq 0}.
\]
Note that $\Gamma^{(0)}E=\Gamma^{(0)}E_{\geq 0}$. We now consider the spectra associated to the $\Gamma$-spaces $\Gamma^{(k)}E_{\geq 0}$. That is, we apply May-Thomason's replacement followed by Segal's functor to obtain a sequence of connective $\Omega$-spectra
\[
\mathbf BW\Gamma^{(k)}E_{\geq 0} \ \ \mathrm{for}\ k \in \N.
\]
Now by lemma \ref{cofibrant} we have the following equivalence
\[
	\mathbf BW\Gamma^{(k)}E_{\geq 0} \xrightarrow{\simeq} \mathbf BW\Omega\Gamma^{(k+1)}E_{\geq 0} \xrightarrow{\simeq}\mathbf B\Omega W\Gamma^{(k+1)}E_{\geq 0}.
\]
By lemma \ref{switch} we have a map $\mathbf B\Omega W\Gamma^{(k)}E\to \Omega \mathbf BW\Gamma^{(k)}E$ which is the identity on $0$-th spaces. In particular, since both spectra are $\Omega$-spectra this map is an equivalence on connective covers. We now observe that since $E(1_+)_{\geq 0}$ is a connective $\Omega$-spectrum, $\Omega \mathbf BW\Gamma^{(k)}E$ is connective for $k\geq 1$ and hence we obtain a stable equivalence
\[
	\mathbf B\Omega W\Gamma^{(k)}E\to \Omega \mathbf BW\Gamma^{(k)}E
\]
for $k \geq 1$. Putting all these maps together we obtain a sequence of connective $\Omega$-spectra $\mathbf BW\Gamma^{(k)}E_{\geq 0}$ together with  stable equivalences
\[
	\mathbf BW\Gamma^{(k)}E_{\geq 0} \xrightarrow{\simeq} \mathbf BW\Omega\Gamma^{(k+1)}E_{\geq 0} \xrightarrow{\simeq}\mathbf B\Omega W\Gamma^{(k+1)}E_{\geq 0}\xrightarrow{\simeq}\Omega \mathbf BW\Gamma^{(k+1)}E_{\geq 0}
\]
that is we have $\mathbf BW\Gamma^{(k)}E_{\geq 0} \xrightarrow{\simeq}\Omega \mathbf BW\Gamma^{(k+1)}E_{\geq 0}$. Thus we are in the situation of lemma \ref{upandacross} and conclude that
\[
\mathbf B W\Gamma^{(0)}E_{\geq 0} = \mathbf B W\Gamma^{(0)}E \simeq E(1_+)_{\geq 0}.
\]
\end{proof}

In light of obtaining the right stable homotopy type, we will from now on assume that we replace a $\Gamma$-space $A$ by $WA$ before applying the functor $\mathbf B$. That is, in what follows $\mathbf BA$ will mean $\mathbf BWA$.

We start with constructing a $\Gamma$-space model for the (connective cover of the) spectrum $\psi_\theta$. Recall that $\psi_\theta$ has as $n$-th space the space $\psi_\theta(\infty,n+1)$ and structure maps given by translation of manifolds in the $(n+1)$ coordinate. The idea is that the spaces $\psi_\theta(\infty,n)$ come with a preferred monoid structure, namely taking disjoint union of manifolds. To make this precise, we introduce the following notation.

\begin{Definition}
Let $\theta:X \to BO(d)$ be a Serre fibration. We obtain for each $m \in \N$ the Serre fibration
\[
\coprod_m \theta: \coprod_m X \to BO(d).
\]
We denote this Serre fibration by $\theta(m_+)$.
\end{Definition}

We can now associate to each $m_+ \in \Gamma^{op}$ the space $\Psi_{\theta(m_+)}(\R^n)$. We think of elements of $\Psi_{\theta(m_+)}(\R^n)$ as manifolds with components labeled by non-basepoint elements of $m_+$ together with $\theta$-structures on those labeled components.

\begin{Lemma}\label{gammatheta}
For all $n \in \N$, the spaces $\Psi_{\theta(m_+)}(\R^n)$ assemble into a $\Gamma$-space.
\end{Lemma}

\begin{proof}
	We have to define the induced maps. Let $\sigma: m_+ \to k_+$ be a map of based sets. We obtain a map
\[
\coprod_{\sigma^{-1}(k_+\setminus \{\ast\})} X \to \coprod_{k_+\setminus \{\ast\}} X. 	
\]
Now define the induced map $\Psi_{\theta(m_+)}(\R^n)\to \Psi_{\theta(k_+)}(\R^n)$ as follows. The image of a pair $(M,l)$ is given by the manifold
\[
M':= l^{-1}\left(\coprod_{\sigma^{-1}(k_+\setminus\{\ast\})}X\right)
\]
together with $\theta(k_+)$-structure given by the composition
\[
M' \xrightarrow{l|_{M'}}\coprod_{\sigma^{-1}(k_+\setminus\{\ast\})}X \to \coprod_{k_+\setminus\{\ast\}}X.
\]
In other words, we relabel the components of $M$ and forget about those components, which get labeled by the basepoint. Taking the empty manifold as basepoint, it is easy to see that this is functorial in $\Gamma^{op}$.
\end{proof}
Note that $\Psi_{\theta(0_+)}\cong \ast$ consisting of only the empy manifold and that we have $\Psi_{\theta(1_+)}(\R^n)= \Psi_\theta(\R^n)$. Also note that we obtain by restriction for any $k \geq 1$ the $\Gamma$-spaces
\[
m_+ \mapsto \psi_{\theta(m_+)}(\infty,k).
\]
As mentionend above, the $\Gamma$-structure can be thought of as taking disjoint union of manifolds. Below we will see that, when stabilizing to $\R^\infty$, taking disjoint union gives a homotopy coherent multiplication on our spaces of manifolds.

\begin{Lemma}
	The spectra $\psi_{\theta(m_+)}$ assemble into a projectively fibrant $\Gamma$-spectrum.
\end{Lemma}

\begin{proof}
	By the above lemma we have for each $n\in \N$ and each map of finite pointed sets $\sigma: m_+ \to k_+$ a map
	\[
	\sigma^n_\ast: \psi_{\theta(m_+)}(\infty,n+1) \to \psi_{\theta(k_+)}(\infty,n+1)
\]
functorial in $\Gamma^{op}$ for fixed $n$. Thus, we have to show that these maps commute with the structure maps, that is we need to show that the diagram
\[
\begin{tikzcd}
	S^1\wedge \psi_{\theta(m_+)}(\infty,n+1) \rar \dar[swap]{id\wedge \sigma^n_\ast} & \psi_{\theta(m_+)}(\infty,n+2)\dar{\sigma^{n+1}_\ast}\\
	S^1\wedge \psi_{\theta(k_+)}(\infty,n+1) \rar & \psi_{\theta(k_+)}(\infty,n+2)
\end{tikzcd}
\]
commutes. But this is clear since the structure maps just translate the manifolds in the $(n+1)$ coordinate, while the map $\sigma^n_\ast$ relabels the components.
\end{proof}

\begin{Definition}
	We denote by $\Gamma \psi_\theta$ the $\Gamma$-spectrum
	\[
	m_+ \mapsto \psi_{\theta(m_+)}.
\]
\end{Definition}

To avoid akward notation, we will denote the induced $\Gamma$-spaces $\Gamma^{(k)}(\Gamma\psi_\theta)$ simply by $\Gamma^{(k)}\psi_\theta$.

\begin{Proposition}\label{proppsi}
	The $\Gamma$-space $\Gamma^{(0)}\psi_\theta$ is a model for the connective cover of $\psi_\theta$, i.e. there is a stable equivalence
	\[
	\mathbf B  \Gamma^{(0)}\psi_\theta \simeq \psi_{\theta, \geq 0}.
\]
\end{Proposition}

\begin{proof}
	We show that $\Gamma \psi_\theta$ is a special $\Gamma$-spectrum. The assertion then follows from lemma \ref{key}.  Since $\psi_{\theta(m_+)}$ is an $\Omega$-spectrum for all $m_+\in \Gamma^{op}$, it suffices to show that $\Gamma^{(k)}\psi_\theta$ is a special $\Gamma$-space.
	
	We observe that the Segal map for $\Gamma^{(k)}\psi_\theta$ 
	\[
	\Gamma^{(k)}\psi_\theta(m_+)\to \prod_m \Gamma^{(k)}\psi_\theta(m_+)
\]
is an embedding and we identify its image with a subspace of the product space. This subspace can be characterized as follows. A tuple $(M_1,\ldots,M_m)$ lies in this subspace if and only if $M_i \cap M_j = \emptyset \subset \R^\infty$ for all $i \neq j$. We show that this subspace is a weak deformation retract of the product space
\[
	\prod_m \Gamma^{(k)}\psi_\theta(m_+)= \prod_m \psi_\theta(\infty,k+1).
\]
To this end, we need a map making manifolds (or more generally any subsets) disjoint inside $\R^\infty$. Consider the maps
\begin{align*}
	F: \R^\infty &\to \R^\infty\\
	(x_1,x_2,\ldots) &\mapsto (0,x_1,x_2,\ldots)
\end{align*}
as well as for any $a\in \R$ the map
\begin{align*}
	G_a: \R^\infty &\to \R^\infty\\
	(x_1,x_2,\ldots) &\mapsto (a+x_1,x_2,\ldots).
\end{align*}
These maps are clearly homotopic to the identity via a straight line homotopy. Choosing $a\in (-1,1)$, the composition $G_a\circ F: \R^\infty \to \R^\infty$ induces a self-map 
\[
\psi_\theta(\infty,k+1)\to \psi_\theta(\infty,k+1)
\]
which is homotopic to the identity. Using for each factor of the product space $\prod_m \psi_\theta(\infty,k+1)$ a different (fixed) real number gives a map
\[
	\prod_m \psi_\theta(\infty,k+1)\to \prod_m \psi_\theta(\infty,k+1)
\]
which is our desired deformation retract, this is also illustrated in figure \ref{fig}.
\begin{figure}
\centering
\def\svgwidth{\columnwidth}
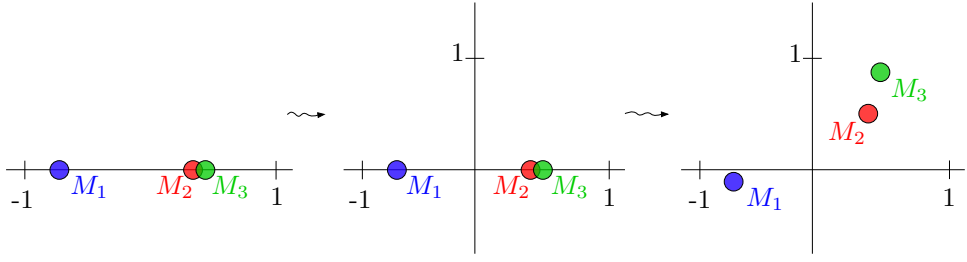
\caption{Making manifolds disjoint}
\label{fig}
\end{figure}
\end{proof}

Recall from lemma \ref{gammatheta} that the association
\[
m_+ \mapsto \Psi_{\theta(m_+)}(\R^n)
\]
defines a $\Gamma$-space for all $n\in \N$.

\begin{Definition}
	Denote by $\Gamma \Psi_\theta$ the (levelwise) colimit of $\Gamma$-spaces
	\[
	\Gamma\Psi_\theta(m_+):= \colim_{n\in \N} \Omega^{n-1}\Psi_{\theta(m_+)}(\R^n).
\]
\end{Definition}
We obtain equivalences
\[
\Gamma^{(0)}\psi_\theta(m_+)= \psi_{\theta(m_+)}(\infty,1) \xrightarrow{\simeq} \colim_{n\in \N} \Omega^{n-1}\Psi_{\theta(m_+)}(\R^n)= \Gamma \Psi_\theta(m_+)
\]
which are clearly functorial in $\Gamma^{op}$. Hence we obtain a levelwise equivalence of $\Gamma$-spaces $\Gamma^{(0)}\psi_\theta\xrightarrow \simeq \Gamma \Psi_\theta$.

\begin{Corollary}
	The $\Gamma$-space $\Gamma \Psi_\theta$ is a model for the connective cover of the spectrum $\psi_\theta$.
\end{Corollary}

We now construct a $\Gamma$-space model for the Madsen-Tilmann spectrum $MT\theta(d)$ and we will show in the next section that this $\Gamma$-space is equivalent to $\Gamma\Psi_\theta$. As before, we will use the Serre fibrations $\theta(m_+)$. First note that the Thom spectrum construction commutes with colimits over $BO(d)$.

\begin{Definition}
	Define the $\Gamma$-spectrum $\Gamma MT \theta(d):\Gamma^{op}\to \catname{Spt}$ by setting
	\[
	\Gamma MT \theta(d)(m_+):= MT \theta(m_+)(d).
\]
For any based map $\sigma: m_+ \to k_+$, define the induced map to be the fold map
\[
	\Gamma MT \theta(d)(m_+) \cong \bigvee_m MT \theta(d) \to \bigvee_k MT \theta(d)\cong \Gamma MT \theta(d)(k_+).
\]
\end{Definition}

As before we will denote the induced $\Gamma$-spaces by $\Gamma^{(k)}MT \theta(d)$ for all $k \in \N$.

\begin{Proposition}\label{propmt}
	The $\Gamma$-space $\Gamma^{(0)}MT \theta(d)$ is a model for the connective cover of the spectrum $MT \theta(d)$.
\end{Proposition}

\begin{proof}
	It again suffices to show that $\Gamma MT \theta(d)$ is special. But this follows easily since in $\catname{Spt}$ we have a stable equivalence
	\[
	MT\theta(m_+)(d)\cong \bigvee_m MT \theta(d) \simeq \prod_m MT\theta(d).
\]
Thus by lemma \ref{key} we obtain a stable equivalence 
\[
\mathbf B \Gamma^{(0)}MT \theta(d) \simeq MT \theta(d)_{\geq 0}.
\]
\end{proof}
Note that by shifting we have a stable equivalence
\[
\mathbf B \Gamma^{(1)}MT \theta(d) \simeq MT \theta(d)[1]_{\geq 0}.
\]

\section{Equivalence of $\Gamma$-space models} 
\label{sec:equivalence_of_gamma_space_models}

In the previous section we have constructed the $\Gamma$-space models $\Gamma \Psi_\theta$ for $\psi_\theta$ and $\Gamma^{(1)}MT\theta (d)$ for $MT \theta(d)[1]_{\geq 0}$. But more is true, by theorem \ref{maingalran} we have for each $m_+ \in \Gamma^{op}$ a weak equivalence of spaces
\[
\Gamma^{(1)}MT\theta(d)(m_+)=\Omega^\infty MT\theta(m_+)(d)[1] \xrightarrow{\simeq} \colim_{n\in \N} \Omega^{n-1}\Psi_{\theta(m_+)}(\R^n)= \Gamma \Psi_\theta(m_+).
\]
The following lemma shows that these equivalences define a levelwise equivalence of $\Gamma$-spaces.

\begin{Lemma}\label{gammaequivalence}
	The weak equivalences of theorem \ref{maingalran2}
	\[
	Th(\theta_n^\ast \gamma_{d,n}^\perp) \xrightarrow{\simeq}\Psi_\theta (\R^n)
\]
assemble into a map of $\Gamma$-spaces. In particular, we obtain a levelwise equivalence
\[
\Gamma^{(1)}MT\theta(d)\xrightarrow{\simeq} \Gamma\Psi_\theta.
\]
\end{Lemma}

\begin{proof}
	We need to show that for any map of based sets $\sigma: m_+ \to k_+$ the diagram
	\[
	\begin{tikzcd}
	Th(\theta_n(m_+)^\ast\gamma_{d,n}^\perp) \rar \dar[swap]{\sigma_\ast} & \Psi_{\theta(m_+)}(\R^n) \dar{\sigma_\ast}\\
	Th(\theta_n(k_+)^\ast\gamma_{d,n}^\perp) \rar  & \Psi_{\theta(k_+)}(\R^n)		
	\end{tikzcd}
\]
commutes. But this follows easily since the left hand vertical map is just the fold map. In particular one can view this map as relabeling components of the wedge and mapping components labeled by $\ast$ to the basepoint. On the other hand this is precisely the description of the right hand vertical map.
\end{proof}

We can now prove the first half of our main theorem.

\begin{Theorem}\label{main1}
	There is an equivalence of spectra
	\[
	MT \theta(d)[1]_{\geq 0} \simeq \psi_{\theta,\geq 0}.
\]
\end{Theorem}

\begin{proof}
	By lemma \ref{gammaequivalence} we have an equivalence of $\Gamma$-spaces
	\[
	\Gamma^{(1)}MT\theta(d)\xrightarrow{\simeq}\Gamma \Psi_\theta.
\]
By proposition \ref{propmt}, $\Gamma^{(1)}MT\theta(d)$ is a model for the spectrum $MT\theta(d)[1]_{\geq 0}$, while by proposition \ref{proppsi} and its corollary the $\Gamma$-space $\Gamma \Psi_\theta$ is a model for the connective cover of $\psi_\theta$. Hence we obtain an equivalence
\[
MT \theta(d)[1]_{\geq 0} \simeq \mathbf B \Gamma^{(1)}MT\theta(d) \simeq \mathbf B\Gamma \Psi_\theta \simeq \psi_{\theta,\geq 0}.
\]
\end{proof}

\section{The cobordism category} 
\label{sec:the_cobordism_category}

In the previous section we have exhibited an equivalence between the connective covers of the spectra $MT\theta(d)[1]$ and $\psi_\theta$. It remains to relate these spectra to the (classifying space of the) topological cobordism category. 

Classically, the cobordism category is a symmetric monoidal category with monoidal product given by taking disjoint union of manifolds. We will see that this is also true for the topological variant in a sense we will make precise below. In particular, having a symmetric monoidal structure endows the classifying space of the cobordism category with the structure of an infinite loop space and we will see that the cobordism category is equivalent as such to the infinite loop space associated to $MT\theta(d)[1]$.

There have appeared several definitions of the cobordism category as a category internal to topological spaces, which all have equivalent classifying spaces. The relevant model for us will be the topological poset model of \cite{galatiusrandal}. We recall its definition. Define the subspace
\[
D_\theta \subset \R\times \psi_\theta(\infty,1)
\]
consisting of pairs $(t,M)$ where $t\in \R$ is a regular value of the projection onto the first coordinate $ M \subset \R \times (-1,1)^\infty \to \R$. Order its elements by $(t,M)\leq (t',M')$ if and only if $t\leq t'$ with the usual order on $\R$ and $M=M'$. 

\begin{Definition}
	The $d$-dimensional cobordism category $\mathrm{Cob}_\theta(d)$ is the topological category associated to the topological poset $D_\theta$. That is, its space of objects is given by $ob\mathrm{(Cob}_\theta(d))=D_\theta$ and its space of morphisms is given by the subspace $mor\text{(Cob}_\theta(d))\subset \R^2 \times \psi_\theta(\infty,1)$ consisting of triples $(t_0,t_1,M)$ where $t_0\leq t_1$. The source and target maps are simply given by forgetting regular values.
\end{Definition}

Taking the internal nerve yields a simplicial space
\[
N_\bullet \mathrm{Cob}_\theta(d): \Delta^{op}\to \catname S.
\]
We will also write $\mathrm{Cob}_\theta(d)$ for the simplicial space obtained from taking the nerve and write $\mathrm{Cob}_\theta(d)_k$ for the space of $k$-simplices.

Considering $\psi_\theta(\infty,1)$ as a constant simplicial space, we have a forgetful map of simplicial spaces $\mathrm{Cob}_\theta(d)\to \psi_\theta(\infty,1)$ defined on $k$-simplices by
\begin{align*}
	\mathrm{Cob}_\theta(d)_k &\to \psi_\theta(\infty,1)\\
	(\underline t,M) &\mapsto M.
\end{align*}

\begin{Theorem}\label{maingal}
	The forgetful map induces a weak equivalence 
	\[
	B\mathrm{Cob}_\theta(d)\xrightarrow{\simeq}\psi_\theta(\infty,1)
\]
where $B\mathrm{Cob}_\theta(d)$ is the realization of the simplicial space $\mathrm{Cob}_\theta(d)$.
\end{Theorem}

\begin{proof}
	\cite[Theorem 3.10]{galatiusrandal}.
\end{proof}

We now encode the symmetric monoidal structure of $\mathrm{Cob}_\theta(d)$ in terms of a $\Gamma$-structure.

\begin{Lemma}
	The simplicial spaces $\mathrm{Cob}_{\theta(m_+)}(d)$ assemble into a $\Gamma$-object in simplicial spaces
	\[
	\mathrm{Cob}_{\theta(-)}(d): \Gamma^{op} \to \catname S^{\Delta^{op}}.
\]
\end{Lemma}

\begin{proof}
	For $m_+ \in \Gamma^{op}$ the $k$-simplices are given as subspaces
	\[
	\mathrm{Cob}_{\theta(m_+)}(d)_k \subset \R^{k+1}\times \psi_{\theta(m_+)}(\infty,1) = \R^{k+1}\times \Gamma^{(0)}\psi_\theta(m_+).
\]
Thus for a map $\sigma: m_+ \to n_+$ we define the map
\[
\mathrm{Cob}_{\theta(m_+)}(d) \to \mathrm{Cob}_{\theta(n_+)}(d)
\]
on $k$-simplices to be induced by the map
\[
id\times \sigma_\ast: \R^{k+1}\times \Gamma^{(0)}\psi_\theta(m_+)\to \R^{k+1}\times \Gamma^{(0)}\psi_\theta(n_+)
\]
where $\sigma_\ast$ comes from the functoriality in $\Gamma^{op}$ of the $\Gamma$-space $\Gamma^{(0)}\psi_\theta$. From this description it is clear that the maps just defined are functorial in $\Delta^{op}$ and hence define a map of simplicial spaces.
\end{proof}

\begin{Definition}
	Denote by $\Gamma \mathrm{Cob}_\theta(d)$ the $\Gamma$-object in simplicial spaces
	\begin{align*}
\Gamma \mathrm{Cob}_{\theta(m_+)}(d) & \to \catname S^{\Delta^{op}}\\
	m_+ &\mapsto \mathrm{Cob}_{\theta(m_+)}(d).
\end{align*}
\end{Definition}

Composing with the realization of simplicial spaces we get a functor
\[
B\Gamma \mathrm{Cob}_\theta(d): \Gamma^{op}\to \catname S.
\]
We obtain a $\Gamma$-space by choosing as basepoints the elements $(\underline 0, \emptyset)\in \mathrm{Cob}_\theta(d)_k$ for all $k\in \N$. 

\begin{Lemma}
	The forgetful map induces a levelwise equivalence of $\Gamma$-spaces
	\[
	B\Gamma \mathrm{Cob}_\theta(d) \xrightarrow{\simeq} \Gamma^{(0)}\psi_\theta.
\]
\end{Lemma}

\begin{proof}
	By construction it is clear that the forgetful maps are functorial in $\Gamma^{op}$ so that they indeed define a map of $\Gamma$-spaces. By theorem \ref{maingal}, these maps are weak equivalences hence we obtain a levelwise eqiuvalence of $\Gamma$-spaces.
\end{proof}

In particular, the $\Gamma$-space $B\Gamma \mathrm{Cob}_\theta(d)$ is very special and applying Segal's functor we obtain a connective $\Omega$-spectrum, which we denote by $\mathbf B \Gamma \mathrm{Cob}_\theta(d)$ to avoid akward notation. In conclusion, we obtain an equivalence of spectra
\[
\mathbf B \Gamma \mathrm{Cob}_\theta(d) \xrightarrow{\simeq} \mathbf B \Gamma^{(0)}\psi_\theta.
\]
Combining with theorem \ref{main1}, we obtain our main theorem.

\begin{thm}
	There are stable equivalence of spectra
	\[
	\mathbf B \Gamma \mathrm{Cob}_\theta(d) \simeq \mathbf B\Gamma^{(0)}\psi_\theta \simeq MT\theta(d)[1]_{\geq 0},
\]
such that the induced equivalences
\[
  \Omega^\infty \mathbf B \Gamma \mathrm{Cob}_\theta (d) \simeq \Omega^\infty \psi_\theta \simeq \Omega^\infty MT \theta(d)[1]
\]
are equivalent to the equivalences of theorem \ref{maingalran} and theorem \ref{maingal}.
\end{thm}
\bibliographystyle{plain}
\bibliography{infiniteloop}
\end{document}

%% file: fig3.pdf_tex
\begingroup%
  \makeatletter%
  \providecommand\color[2][]{%
    \errmessage{(Inkscape) Color is used for the text in Inkscape, but the package 'color.sty' is not loaded}%
    \renewcommand\color[2][]{}%
  }%
  \providecommand\transparent[1]{%
    \errmessage{(Inkscape) Transparency is used (non-zero) for the text in Inkscape, but the package 'transparent.sty' is not loaded}%
    \renewcommand\transparent[1]{}%
  }%
  \providecommand\rotatebox[2]{#2}%
  \ifx\svgwidth\undefined%
    \setlength{\unitlength}{984.61552018bp}%
    \ifx\svgscale\undefined%
      \relax%
    \else%
      \setlength{\unitlength}{\unitlength * \real{\svgscale}}%
    \fi%
  \else%
    \setlength{\unitlength}{\svgwidth}%
  \fi%
  \global\let\svgwidth\undefined%
  \global\let\svgscale\undefined%
  \makeatother%
  \begin{picture}(1,0.26324996)%
    \put(0,0){\includegraphics[width=\unitlength,page=1]{fig3.pdf}}%
    \put(0.00878384,0.04803276){\makebox(0,0)[lb]{\smash{-1}}}%
    \put(0.278093,0.05070376){\makebox(0,0)[lb]{\smash{1}}}%
    \put(0,0){\includegraphics[width=\unitlength,page=2]{fig3.pdf}}%
    \put(0.35978378,0.04827184){\makebox(0,0)[lb]{\smash{-1}}}%
    \put(0.62359413,0.0502256){\makebox(0,0)[lb]{\smash{1}}}%
    \put(0,0){\includegraphics[width=\unitlength,page=3]{fig3.pdf}}%
    \put(0.71317452,0.04779368){\makebox(0,0)[lb]{\smash{-1}}}%
    \put(0.97667472,0.0502256){\makebox(0,0)[lb]{\smash{1}}}%
    \put(0,0){\includegraphics[width=\unitlength,page=4]{fig3.pdf}}%
    \put(0.81696632,0.20051254){\makebox(0,0)[lb]{\smash{1}}}%
    \put(0.4657442,0.20170793){\makebox(0,0)[lb]{\smash{1}}}%
    \put(0,0){\includegraphics[width=\unitlength,page=5]{fig3.pdf}}%
    \put(0.06138036,0.05162375){\color[rgb]{0,0,0}\makebox(0,0)[lb]{\smash{}}}%
    \put(0.21656141,0.06268305){\color[rgb]{0,0.80392157,0}\makebox(0,0)[lb]{\smash{$M_3$}}}%
    \put(0.0699337,0.06225153){\color[rgb]{0,0,1}\makebox(0,0)[lb]{\smash{$M_1$}}}%
    \put(0.15876198,0.06256063){\color[rgb]{1,0,0}\makebox(0,0)[lb]{\smash{$M_2$}}}%
    \put(0.4188866,0.06256266){\color[rgb]{0,0,1}\makebox(0,0)[lb]{\smash{$M_1$}}}%
    \put(0.50879668,0.06264774){\color[rgb]{1,0,0}\makebox(0,0)[lb]{\smash{$M_2$}}}%
    \put(0.56955445,0.06208247){\color[rgb]{0,0.80392157,0}\makebox(0,0)[lb]{\smash{$M_3$}}}%
    \put(0.77268126,0.05037316){\color[rgb]{0,0,1}\makebox(0,0)[lb]{\smash{$M_1$}}}%
    \put(0.8579696,0.11943562){\color[rgb]{1,0,0}\makebox(0,0)[lb]{\smash{$M_2$}}}%
    \put(0.92299379,0.16436221){\color[rgb]{0,0.80392157,0}\makebox(0,0)[lb]{\smash{$M_3$}}}%
  \end{picture}%
\endgroup%